\documentclass[10pt,twoside,reqno]{amsart}
\usepackage{amsmath}
\usepackage{amsfonts}
\usepackage{amssymb}
\usepackage{color}
\usepackage{mathrsfs}
\usepackage{cite}
\usepackage{times}
\newcommand{\R}{\mathbb R}

\newcommand{\ve}{\varepsilon} 
\newcommand{\rd}{\mathrm{d}}

\newcommand{\bqn}{\begin{equation}}
\newcommand{\eqn}{\end{equation}}
\newcommand{\bean}{\begin{eqnarray}}
\newcommand{\eean}{\end{eqnarray}}
\DeclareMathAlphabet{\mathpzc}{OT1}{pzc}{m}{it}
\newtheorem{theorem}{Theorem}[section]
\newtheorem{corollary}[theorem]{Corollary}

\newtheorem{proposition}[theorem]{Proposition}

\numberwithin{equation}{section}

\title{Strong Solutions to a Nonlocal-in-Time Semilinear Heat Equation}

\begin{document}

\author{Christoph Walker}
\address{Leibniz Universit\"at Hannover\\ Institut f\" ur Angewandte Mathematik \\ Welfengarten 1 \\ D--30167 Hannover\\ Germany}
\email{walker@ifam.uni-hannover.de}
%
%
\date{\today}
%
%
\begin{abstract}
Existence of strong solutions to a nonlocal semilinear heat equation is shown. The main feature of the equation is that the nonlocal term depends on the unknown on the whole time interval of existence, the latter being given a priori. The proof relies on Schauder's fixed point theorem and semigroup theory.
\end{abstract}

%
%

\date{\today}

\maketitle

\section{Introduction}

This note is dedicated to the nonlocal problem
\bqn\label{1}
\partial_t u-\Delta u+\varphi\left(\int_0^T u(s)\, \rd s\right) u=0\,,\quad (t,x)\in (0,T]\times\Omega\,,\quad u\vert_{\partial\Omega}=0\,,\quad u(0)=u^0\,,
\eqn
with given potential $\varphi$, given initial datum $u^0$, and given existence time $T>0$, and 
where $\Omega\subset\R^n$, $n\ge 1$, is a bounded $C^2$-domain. Equation~\eqref{1} arises in the modeling of a biological nanosensor in the chaotic dynamics of a polymer chain in an aqueous solution and has been introduced and considered in \cite{StaraStara17,Stara18,Stara20}. We refer to these papers for more information on the modeling background.

Clearly, the main feature of \eqref{1} is that the nonlinearity depends on
$$
u_T:=\int_0^T u(s)\, \rd s\,,
$$
that is, on the unknown on the {\it whole interval of existence $[0,T]$}, where $T>0$ is {\it a priori} given and not free to be chosen. The problem itself is thus {\it not} an evolution problem in the usual sense as it does not satisfy the {\it Volterra property} since a solution at a time instant depends also on later times, i.e. on the future. Equations violating the Volterra property arise, of course, also in other contexts, e.g.  in certain reaction-diffusion equations with non-local initial conditions~\cite{Pao95,Ly15}, in models for long-term weather forecast (see \cite{Shelukin93} and the references therein), or in the study of stationary solutions to population models including age- and spatial structure (e.g. see \cite{ChWCrelle11, ChWJEPE18}) to name but a few.\\

In \cite{Stara20} (see also \cite{Stara18}) the existence of weak solutions to \eqref{1} is shown assuming a non-negative continuous potential $\varphi$ for which $s\mapsto\varphi(s)s$ is differentiable and non-decreasing. Herein we provide a simple and short proof for the existence of {\it strong} solutions under fairly general assumptions on the potential $\varphi$. In particular, we do not use a differentiability or monotonicity assumption. More precisely, we shall prove the following result:

\begin{theorem}\label{T1}
Let $\varphi\in C(\R,\R)$ be non-negative and $2p>n$. If either
\begin{itemize}
\item[(i)] $u^0\in L_\infty(\Omega)$
\end{itemize}
or
\begin{itemize}
\item[(ii)] $\varphi(s)\le a (1+\vert s\vert)$, $s\in\R$, for some $a>0$ and $u^0\in L_p(\Omega)$,
\end{itemize}
then there is at least one strong solution 
$$
u\in C\big([0,T],L_p(\Omega)\big)\cap C^1\big((0,T],L_p(\Omega)\big)\cap C\big((0,T],W_{p}^2(\Omega)\big)
$$
to \eqref{1}. 
Moreover, $\| u(t)\|_p\le \|u^0\|_p$ for $t\in [0,T]$. If $u^0\ge 0$, then $u(t)\ge 0$ for~$t\in [0,T]$.
\end{theorem}

The proof is an application of Schauder's fixed point theorem based on the fact that, under suitable assumptions, the operator $A(u_T):=-\Delta_D+\varphi(u_T)$ generates a semigroup $(e^{-t A(u_T)})_{t\ge 0}$ on $L_p(\Omega)$, where $-\Delta_D$ denotes the Laplacian subject to Dirichlet boundary conditions.
Solutions to \eqref{1} are thus of the form
$$
u(t)=e^{-t A(u_T)} u^0\,,\quad t\in [0,T]\,,
$$
and, consequently, $u_T$ satisfies the fixed point equation
\bqn\label{uT}
u_T=\int_0^T e^{-t A(u_T)} u^0\,\rd t\,.
\eqn
As for $u_T$, we derive further information.

\begin{corollary}\label{C1}
Let $u$ be the solution to \eqref{1} provided by Theorem~\ref{T1}. If $p\ge 2$, then
$$
\| \nabla u_T\|_2^2+\int_\Omega \varphi(u_T) \vert u_T\vert^2\,\rd x =\int_\Omega \big(u^0-u(T)\big)\, u_T\,\rd x \le 2T\| u^0\|_2^2\,.
$$
If $u^0\in W_{p}^{\alpha}(\Omega)$ for some $\alpha>0$, then $u_T\in W_{p}^2(\Omega)$ satisfies
\bqn\label{p2}
-\Delta u_T+\varphi(u_T) u_T=u^0- u(T)\quad \text{in}\quad \Omega\,,\qquad u_T=0\quad \text{on}\quad  \partial\Omega\,.
\eqn
\end{corollary}

In the next section we shall prove Theorem~\ref{T1}. As an immediate consequence we obtain Corollary~\ref{C1}. We shall also prove in Proposition~\ref{P1} the uniqueness of solutions to \eqref{1} for small initial values or small maximal existence time provided the potential~$\varphi$ is Lipschitz.

\section{Proof of Theorem~\ref{T1}}

We let $\mathcal{L}(E,F)$ denote the space of bounded linear operators between two Banach spaces $E$ and $F$ with norm \mbox{$\|\cdot\|_{\mathcal{L}(E,F)}$} and $\mathcal{L}(E):=\mathcal{L}(E,E)$. The norm e.g. in $E$ is denoted by $\|\cdot\|_{E}$ and $\|\cdot\|_{q}:=\|\cdot\|_{L_q(\Omega)}$ for $q\in [1,\infty]$. Given $p\in (1,\infty)$ we use the notation
\begin{equation*}
W_{p,D}^{\alpha}(\Omega):=\left\{ 
\begin{array}{lcl} 
\{u\in  W_{p}^{\alpha}(\Omega)\,;\, u=0\ \text{on}\ \partial\Omega\} & \text{ if } & \alpha \in \left( {\frac{1}{p}},2 \right]\ ,\\
& & \\
W_{p}^{\alpha}(\Omega) & \text{ if } & 0\le \alpha < {\frac{1}{p}}\,.
\end{array}
\right. 
\end{equation*}
Moreover, for $\omega>0$ and $\kappa\ge 1$,  let
$\mathcal{H}(W_{p,D}^2(\Omega),L_p(\Omega);\kappa,\omega)$ be the set of all $\mathcal{A}\in\mathcal{L}(W_{p,D}^2(\Omega),L_p(\Omega))$ such that $\omega+\mathcal{A}$ is an isomorphism from $W_{p,D}^2(\Omega)$ onto $L_p(\Omega)$ satisfying the resolvent estimates
$$
\frac{1}{\kappa}\,\le\,\frac{\|(\mu+\mathcal{A})z\|_{L_p(\Omega)}}{\vert\mu\vert \,\| z\|_{L_p(\Omega)}+\|z\|_{W_{p,D}^2(\Omega)}}\,\le \, \kappa\ ,\quad Re\, \mu\ge \omega\ ,\quad z\in W_{p,D}^2(\Omega)\setminus\{0\}\ .
$$
Then $\mathcal{A}\in\mathcal{H}(W_{p,D}^2(\Omega),L_p(\Omega);\kappa,\omega)$ implies that 
$\mathcal{A}\in\mathcal{H}(W_{p,D}^2(\Omega),L_p(\Omega))$; that is, 
$-\mathcal{A}$  generates an analytic semigroup $(e^{-t\mathcal{A}})_{t\ge 0}$ on $L_p(\Omega)$ with domain $W_{p,D}^2(\Omega)$, see \cite[I.Theorem 1.2.2]{LQPP}.
Recall that $-\Delta_D\in \mathcal{H}(W_{p,D}^2(\Omega),L_p(\Omega))$.\\

\subsection*{Proof of Theorem~\ref{T1}, Part \bf (i)} Assume that $\varphi\in C(\R,\R)$ is non-negative and $u^0\in L_\infty(\Omega)$. Since $\varphi$ is then uniformly continuous and bounded on bounded sets, it follows that (considered as Nemytskii operator)
\bqn\label{300}
\varphi\in C\big(L_\infty(\Omega),L_\infty(\Omega)\big)\  \text{ is bounded on bounded sets}\,.
\eqn
Set $S_0:=T\|u^0\|_\infty$ and let
$$
X_T:=\bar{\mathbb{B}}_{L_\infty(\Omega)}(0,S_0)
$$
denote the closed ball in $L_\infty(\Omega)$ of radius $S_0$ centered at the origin. Fix $2p\in (n,\infty)$ and note that, given any $u_T\in X_T$, the mapping $\varphi(u_T):=[w\mapsto \varphi(u_T) w]\in \mathcal{L}\big(L_p(\Omega)\big)$ obviously satisfies
\bqn\label{op}
\|\varphi(u_T)\|_{\mathcal{L}(L_p(\Omega))}\le \|\varphi(u_T)\|_\infty\le \max_{[-S_0,S_0]}\varphi\,,\quad u_T\in X_T\,.
\eqn
We now infer from the fact that $-\Delta_D\in \mathcal{H}(W_{p,D}^2(\Omega),L_p(\Omega))$ and the perturbation result \cite[I.Theorem 1.3.1]{LQPP} that
\bqn\label{301}
A(u_T):=-\Delta_D+\varphi(u_T)\in\mathcal{H}(W_{p,D}^2(\Omega),L_p(\Omega);\kappa,\omega(S_0))
\eqn
for some $\omega(S_0)>0$ and $\kappa\ge 1$. Moreover, since $\varphi$ is non-negative, we have that $-A(u_T)$ generates a positive contraction semigroup $(e^{-t A(u_T)})_{t\ge 0}$ on each $L_q(\Omega)$ for $q\in (1,\infty]$ (which, however, is {\it not} strongly continuous for $q=\infty$), hence
\bqn\label{cs}
\| e^{-t A(u_T)}\|_{\mathcal{L}(L_q(\Omega))}\le 1\,,\quad t\ge 0\,,\quad q\in (1,\infty]\,.
\eqn
Now, let us define
\bqn\label{Phi}
\Phi(u_T):=\int_0^T e^{-t A(u_T)} u^0\,\rd t\,,\quad u_T\in X_T\,.
\eqn
Then \eqref{cs} implies that
$$
\|\Phi(u_T)\|_\infty\le \int_0^T \|e^{-t A(u_T)}\|_{\mathcal{L}(L_\infty(\Omega))}\, \|u^0\|_\infty\,\rd t \le T\|u^0\|_\infty =S_0\,,\quad u_T\in X_T\,,
$$
so that $\Phi:X_T\rightarrow X_T$. Choose $2\alpha\in (n/p,2)$ and note that \eqref{301} together with \cite[II.Lemma 5.1.3]{LQPP} yield that 
there are $M(S_0)\ge 1$ and $\nu(S_0)>0$ such that
\bqn\label{smoot}
\|e^{-t A(u_T)}\|_{\mathcal{L}(L_p(\Omega),W_{p,D}^{2\alpha}(\Omega))}\le M(S_0)e^{\nu(S_0) t} t^{-\alpha}\,,\quad t>0\,,\quad u_T\in X_T\,.
\eqn
Therefore,
\bqn\label{smoott}
\|\Phi(u_T)\|_{W_{p,D}^{2\alpha}(\Omega)}\le \int_0^T \|e^{-t A(u_T)}\|_{\mathcal{L}(L_p(\Omega),W_{p,D}^{2\alpha}(\Omega))}\, \|u^0\|_{p}\,\rd t \le \frac{M(S_0)}{1-\alpha} e^{\nu(S_0) T} T^{1-\alpha} \|u^0\|_{p}\le c(S_0) \,,
\eqn
and we conclude that $\Phi(X_T)$ is bounded in $W_{p,D}^{2\alpha}(\Omega)$, the latter being compactly embedded in $C(\bar\Omega)$ since $2\alpha>n/p$. 

In order to prove the continuity of $\Phi$ let us observe that, given $u_T, v_T\in X_T$, we have
\begin{align}
e^{-t A(u_T)}-e^{-t A(v_T)} &=-\int_0^t \frac{\rd }{\rd s} e^{-(t-s) A(u_T)} e^{-s A(v_T)}\,\rd s\notag\\
&= -\int_0^t e^{-(t-s) A(u_T)}\,\big(\varphi(u_T)-\varphi(v_T)\big)\, e^{-s A(v_T)}\,\rd s\,,\label{form}
\end{align}
so that, using \eqref{op}, \eqref{cs}, and \eqref{smoot},
\begin{align}
\|e^{-t A(u_T)}&-e^{-t A(v_T)}\|_{\mathcal{L}(L_p(\Omega),W_{p,D}^{2\alpha}(\Omega))}\notag\\
 &\le \int_0^t \|e^{-(t-s) A(u_T)}\|_{\mathcal{L}(L_p(\Omega),W_{p,D}^{2\alpha}(\Omega))}\, \|\varphi(u_T)-\varphi(v_T)\|_{\mathcal{L}(L_p(\Omega))}\, \| e^{-s A(v_T)}\|_{\mathcal{L}(L_p(\Omega))}\,\rd s\\
&\le c(S_0)\, e^{c(S_0) T} \, t^{1-\alpha} \, \|\varphi(u_T)-\varphi(v_T)\|_{\infty}\,.\label{lll}
\end{align}
Due to the continuous embedding of  $W_{p,D}^{2\alpha}(\Omega)$ in $C(\bar\Omega)$ we derive that
\begin{equation*}
\begin{split}
\|\Phi(u_T)-\Phi(v_T)\|_\infty&\le c\, \|\Phi(u_T)-\Phi(v_T)\|_{W_{p,D}^{2\alpha}(\Omega)}\le
c \int_0^T \|e^{-t A(u_T)}-e^{-t A(v_T)}\|_{\mathcal{L}(L_p(\Omega),W_{p,D}^{2\alpha}(\Omega))} \|u^0\|_p\,\rd t\\
&\le c_1(S_0) \|\varphi(u_T)-\varphi(v_T)\|_{\infty}
\end{split}
\end{equation*}
for $u_T, v_T\in X_T$, hence the continuity of $\Phi:X_T\rightarrow X_T$ due to \eqref{300}. Consequently, $\Phi\in C(X_T,X_T)$ with precompact image so that Schauder's fixed point theorem implies the existence of $u_T\in X_T$ such that $u_T=\Phi(u_T)$. We define
\bqn\label{u}
u(t):=e^{-t A(u_T)} u^0\,,\quad t\in [0,T]\,,
\eqn
in order to obtain a solution to \eqref{1}. If $u^0\ge 0$, then $u(t)\ge 0$ for $t\in [0,T]$ since the semigroup is positive. This yields part~(i) of Theorem~\ref{T1}.\\

\subsection*{Proof of Theorem~\ref{T1}, Part \bf (ii)}  Now  assume that $\varphi\in C(\R,\R)$ is non-negative and that $\varphi(s)\le a (1+\vert s\vert)$, $s\in\R$, for some $a>0$. Consider $u^0\in L_p(\Omega)$ with $2p>n$. We then adapt the prove above. Set now $S_0:=T\|u^0\|_p$ and
$$
X_T:=\bar{\mathbb{B}}_{L_p(\Omega)}(0,S_0)\,.
$$
Note that the assumptions on $\varphi$ entail
\bqn\label{1A}
\varphi\in BC\big(X_T,L_p(\Omega)\big)\,.
\eqn
Moreover, $W_{p,D}^{2-2\ve}(\Omega)$ embeds continuously in $C(\bar\Omega)$ for $\ve>0$ small enough since $2p>n$. Thus,  for $w\in W_{p,D}^2(\Omega)$ we have
\begin{align}
\|\varphi(u_T)w\|_p&\le c\, \|\varphi(u_T)\|_p\, \| w\|_{W_{p,D}^{2-2\ve}(\Omega)}\label{e1}\\
&\le c\,\|\varphi(u_T)\|_p\, \| w\|_{p}^{\ve}\, \| w\|_{W_{p,D}^2(\Omega)}^{1-\ve} \notag\\
&\le c(\delta) \,\|\varphi(u_T)\|_p^{1/\ve}\|w\|_p+\delta\|w\|_{W_{p,D}^2(\Omega)}\notag
\end{align}
with $\delta>0$ small enough. Hence, \eqref{1A} and \cite[I.Theorem 1.3.1]{LQPP} ensure that \eqref{301} holds true again. Moreover, \eqref{cs} holds for $q=p$. Defining $\Phi$ as in  \eqref{Phi}, we argue as in part (i) to deduce that $\Phi\in C(X_T, X_T)$ has a precompact image, where the continuity follows from the fact that \eqref{form} along with \eqref{cs}, \eqref{smoot}, and \eqref{e1} ensure
\begin{equation*}
\begin{split}
\|e^{-t A(u_T)}&-e^{-t A(v_T)}\|_{\mathcal{L}(L_p(\Omega))}\\
 &\le \int_0^t \|e^{-(t-s) A(u_T)}\|_{\mathcal{L}(L_p(\Omega))}\, \|\varphi(u_T)-\varphi(v_T)\|_{\mathcal{L}(W_{p,D}^{2-2\ve}(\Omega),L_p(\Omega))}\, \| e^{-s A(v_T)}\|_{\mathcal{L}(L_p(\Omega),W_{p,D}^{2-2\ve}(\Omega))}\,\rd s\\
&\le c(S_0) \, \|\varphi(u_T)-\varphi(v_T)\|_{p}\,.
\end{split}
\end{equation*}
The assertion then again follows by applying Schauder's fixed point theorem. 

\section{Proof of Corollary~\ref{C1}}

Let $u$ be the solution to \eqref{1} provided by Theorem~\ref{T1}. If $p\ge 2$, then
$$
\int_s^T u(\sigma)\,\rd \sigma\in W_{2,D}^2(\Omega)
$$
for $s\in (0,T)$ due to the regularity of the solution $u$. Testing \eqref{1} by this quantity and letting then $s\rightarrow 0^+$ yields
\bqn\label{p1}
\| \nabla u_T\|_2^2+\int_\Omega \varphi(u_T) \vert u_T\vert^2\,\rd x =\int_\Omega  \big(u^0-u(T)\big)\, u_T\,\rd x \le \| u^0-u(T)\|_2\,\| u_T\|_2\,.
\eqn
Now, \eqref{cs} along with \eqref{u} respectively \eqref{uT} entail
$$
\|u(T)\|_2\le \| u^0\|_2\,,\qquad \|u_T\|_2\le T\| u^0\|_2\,,
$$
hence
$$
\| \nabla u_T\|_2^2+\int_\Omega \varphi(u_T) \vert u_T\vert^2\,\rd x\le 2T\| u^0\|_2^2\,.
$$
If $u^0\in W_p^{\alpha}(\Omega)$ for some $\alpha>0$, then~\eqref{uT} implies $u_T\in W_{p,D}^2(\Omega)$
since 
$$
\|e^{-t A(u_T)}\|_{\mathcal{L}(W_{p,D}^{\alpha}(\Omega),W_{p,D}^{2}(\Omega))}\le M(S_0)e^{\nu(S_0) t} t^{\alpha/2-1}\,,\quad t>0\,,
$$
due to \cite[II.Lemma 5.1.3]{LQPP}. 
Integrating \eqref{1} with respect to $t\in (0,T)$ then gives \eqref{p2}. 
This proves Corollary~\ref{C1}.

\section{Uniqueness for Small Data}

As noticed in \cite{Stara20} one can prove the uniqueness of solutions to~\eqref{1} if $S_0:=T\|u^0\|_\infty$ is small provided that, in addition, 
\bqn\label{Lip}
\varphi:\R\rightarrow\R\ \text{ is locally Lipschitz continuous}
\eqn
and 
\bqn\label{mon}
s\mapsto \varphi(s)s \ \text{ is non-decreasing}\,.
\eqn
Indeed, suppose the conditions of Theorem~\ref{T1}~(i) with $p\ge 2$ and let \eqref{Lip} and \eqref{mon} hold true. Consider two solutions $u$ and $v$ to~\eqref{1} with $u(0)=v(0)=u^0\in L_\infty(\Omega)$ and set $w_T:=u_T-v_T$. Then, 
\bqn\label{lip1}
\|\varphi(u_T)-\varphi(v_T)\|_2\le L(S_0)\,\| u_T-v_T\|_2 = L(S_0)\,\| w_T\|_2
\eqn
for some constant $L(S_0)$ since \eqref{Lip} implies that $\varphi$ is uniformly Lipschitz on the set $[-S_0,S_0]$.
The same argument leading to \eqref{p1} entails that
\begin{equation}\label{x1}
\begin{split}
\| \nabla w_T\|_2^2&=-\int_\Omega \big(\varphi(u_T) u_T-\varphi(v_T) v_T\big)\,\big(u_T-v_T\big) \,\rd x+\int_\Omega \big(v(T)-u(T)\big)\, w_T\,\rd x\\
&\le \|v(T)-u(T)\|_2\, \|w_T\|_2\,,
\end{split}
\end{equation}
where \eqref{mon} allows us to conclude the inequality.
Now, owing to \eqref{form}, \eqref{cs}, and \eqref{lip1} we have
\begin{equation}\label{es}
\begin{split}
\|v(T)-u(T)\|_2 &=\big\|\big(e^{-t A(u_T)}-e^{-t A(v_T)}\big) u^0 \big\|_2\\
&  \le \int_0^T \big\| e^{-(t-s) A(u_T)}\big\|_{\mathcal{L}(L_2(\Omega))}\, \|\varphi(u_T)-\varphi(v_T)\|_2\, \big\|e^{-s A(v_T)}\big\|_{\mathcal{L}(L_\infty(\Omega))}\, \|u^0\|_\infty\,\rd s\\
&\le S_0\, L(S_0)\,\| w_T\|_2 \,.
\end{split}
\end{equation}
Therefore, \eqref{x1} and \eqref{es} entail that
\begin{equation*}
\begin{split}
\| \nabla w_T\|_2^2
&\le S_0\, L(S_0)\,\| w_T\|_2^2\le c(\Omega)\,  S_0\, L(S_0)\,\|\nabla w_T\|_2^2\,,
\end{split}
\end{equation*}
where $c(\Omega)$ is the constant from Poincar\'e's inequality. Now, if $c(\Omega) S_0 L(S_0)<1$, this implies that $w_T\equiv 0$ in $\Omega$, hence $u_T=v_T$ which readily implies that the solutions $u$ and $v$ coincide.\\

However, the monotonicity condition \eqref{mon} is not needed to prove uniqueness of solutions to \eqref{1} for small data as is shown in the next proposition. Moreover, in order to prove existence and uniqueness of solutions to~\eqref{1} we also do not need a sign or growth condition on $\varphi$ in this particular case and only impose the Lipschitz condition~\eqref{Lip}.

\begin{proposition}\label{P1}
Let $\varphi$ satisfy~\eqref{Lip}, let $2p>n$, and consider $u^0\in L_p(\Omega)$. There is $R>0$ such that~\eqref{1} has a unique solution 
$$
u\in C\big([0,T],L_p(\Omega)\big)\cap C^1\big((0,T],L_p(\Omega)\big)\cap C\big((0,T],W_{p}^2(\Omega)\big)
$$
provided that $T \|u^0\|_p\le R$.
\end{proposition}

\begin{proof}
We use Banach's fixed point theorem. For this fix $2\alpha\in (n/p,2)$, let $S_0>0$, and put
$$
X_T:=\bar{\mathbb{B}}_{W_{p,D}^{2\alpha}(\Omega)}(0,S_0)\,.
$$
Note that $W_{p,D}^{2\alpha}(\Omega)$ embeds continuously into $C(\bar\Omega)$. 
Thus, since $\varphi$ is uniformly Lipschitz continuous on compact sets, there is a constant $L(S_0)>0$ such that $\varphi$ (considered as Nemytskii operator) satisfies
\bqn\label{100}
\|\varphi(u_T)-\varphi(v_T)\|_\infty \le L(S_0)\,\| u_T-v_T\|_{W_{p,D}^{2\alpha}(\Omega)}\,,\quad u_T, v_T\in X_T\,.
\eqn 
In particular, since $\|\varphi(u_T)\|_\infty\le c(S_0)$, it follows as in the proof of part (i) of Theorem~\ref{T1} that  \cite[I.Theorem 1.3.1]{LQPP} and \cite[II.Lemma 5.1.3]{LQPP} imply \eqref{301} and \eqref{smoot}. The latter entails, as in \eqref{smoott}, that
\bqn\label{smoott1}
\|\Phi(u_T)\|_{W_{p,D}^{2\alpha}(\Omega)} \le \frac{M(S_0)}{1-\alpha} e^{\nu(S_0) T} T^{1-\alpha} \|u^0\|_{p} \,.
\eqn
Moreover, \eqref{lll} along with \eqref{100} also yield
\begin{align*}
\|&e^{-t A(u_T)}-e^{-t A(v_T)}\|_{\mathcal{L}(L_p(\Omega),W_{p,D}^{2\alpha}(\Omega))}\le 
c(S_0)\, e^{c(S_0) T} \, t^{1-\alpha} \,
 \| u_T-v_T\|_{W_{p,D}^{2\alpha}(\Omega)}\,,\quad u_T, v_T\in X_T\,,
\end{align*}
for some $c(S_0)>0$. Therefore
\begin{align*}
\|\Phi(u_T)-\Phi(v_T)\|_{W_{p,D}^{2\alpha}(\Omega)}
 &\le \int_0^T \| e^{-t A(u_T)}-e^{-t A(v_T)}\|_{\mathcal{L}(L_p(\Omega),W_{p,D}^{2\alpha}(\Omega))}\,\|u^0\|_p \,\rd t\\
&\le c_1(S_0) \, e^{c(S_0)T}\,T^{2-\alpha}\,\|u^0\|_p \, \| u_T-v_T\|_{W_{p,D}^{2\alpha}(\Omega)} \,.
\end{align*}
Together with \eqref{smoott1} this shows that $\Phi:X_T\rightarrow X_T$ is a contraction provided that $T\|u^0\|_p$ is small enough. Consequently, there exists a unique $u_T\in X_T$ with $\Phi(u_T)=u_T$ if $T\|u^0\|_p$ is small enough. 
\end{proof}

\end{document}